\documentclass[a4paper,12pt]{article}
\usepackage[utf8]{inputenc}
\usepackage[left=2.2cm,right=2.2cm, top=2cm,bottom=2cm,bindingoffset=0cm, marginparwidth=3cm]{geometry}
\usepackage{amssymb, amsmath, amsthm, verbatim, tikz, wrapfig, hyperref}
\newcounter{theorems}
\numberwithin{theorems}{section}

\newtheorem{thm}[theorems]{Theorem}

\newtheorem{lem}[theorems]{Lemma}
\newtheorem{defin}[theorems]{Definition}
\newtheorem{assu}[theorems]{Assumption}
\newtheorem*{thm*}{Theorem}
\newtheorem*{rem*}{Remark}

\newtheorem*{que*}{Question}
\newtheorem*{corB}{Corollary B}
\newtheorem*{Theorem A}{Theorem A}
\def\om {\omega }
\def\oms {\omega_{stat} }
\def\As {A_{stat} }
\def\RR {\mathbb{R}}

\def\eps {\varepsilon}
\def\sub {\subset}
\def\be {\begin{equation}}
\def\ee {\end{equation}}
\def\bp {\begin{proof}}
\def\ep {\end{proof}}

\def\ss {S^1}
\def\ms {P^r_{MS}}
\def\m {^{-1}}
\def\A {A_M}
\def\nw {\Omega}
\def\wu {W^u}
\def\ws {W^s}

\def\cu {\overline{W^u}}
\def\cs {\overline{W^s}}
\def\P {P^{r}}
\def\leb {Leb}
\DeclareMathOperator{\dist}{dist}
\DeclareMathOperator{\Diff}{Diff}

\DeclareMathOperator{\per}{per}
\def \arr {\leadsto}

\title{Milnor Attractors of Skew Products with the Fiber
a Circle\footnote{The 
final publication is available at Springer via \url{http://dx.doi.org/10.1007/s10883-016-9334-7}}}
\author{A. Okunev\footnote{National Research University Higher School of Economics; supported by part by a grant of the Simons Foundation;}}

\begin{document}
\maketitle
\begin{abstract}
We prove that for a generic skew product with circle fiber over an Anosov diffeomorphism the Milnor attractor (also called the likely limit set) coincides with the statistical attractor, is Lyapunov stable, and either has zero Lebesgue measure or coincides with the whole phase space. 
As a consequence we conclude that such skew product is either transitive or has non-wandering set of zero measure.
The result is proved under the assumption that the fiber maps preserve the orientation of the circle, and the skew product is partially hyperbolic.
\end{abstract}

\section{Introduction}
This paper is motivated by the following open questions by Yu.S. Ilyashenko:
\begin{itemize}
\item Is there an open set of diffeomorphisms with Lyapunov unstable attractor?~(\cite{ISh})

\item Is there an open set of diffeomorphisms with thick (i.e. having positive, but not full, Lebesgue measure) attractor?
\end{itemize}
The word ``attractor'' here is usually understood as ``Milnor attractor'' (see definition~\ref{d:AM}), but there are many other nonequivalent definitions of attractors.

A breakthrough in the study of
the first question was recently done by I. Shilin (\cite{Shi}, in preparation). He established local topological genericity of diffeomorphisms with Lyapunov unstable Milnor attractor.

A positive answer to the second question was obtained by Yu. Ilyashenko \cite{Ily} for boundary preserving  diffeomorphisms of manifolds with boundary.

In this paper, the questions stated above are studied for a particular class of diffeomorphisms of closed manifolds, namely, for partially hyperbolic skew products whose central fibers are circles.

Skew products with one-dimensional fibers form an important class of dynamical systems. On one hand, this class has numerous interesting properties: attractors with intermingled basins~\cite{Kan}, bony~\cite{Kud},~\cite{Kud2} and thick~\cite{Ily} attractors and so on. On the other hand, this class is relatively simple.

Non boundary preserving skew products with the fiber a segment were studied by V. Kleptsyn and D. Volk~\cite{KV},~\cite{KV2}. For such skew products there is a finite collection of attracting and repelling invariant sets, each of them is a so-called bony graph.  A bony graph is almost a graph of a function from the base to the fiber, but some points in the base correspond not to a single point, but to an interval in the fiber. The attractor has zero measure. Lyapunov stability of the attractor remained an open question.

In this paper we give a negative answer to both questions we started with for orientation-preserving skew products with circle fiber. We prove that typically the Milnor attractor is Lyapunov stable and not thick, and coincides with the statistical attractor (see definition~\ref{d:As}). As an easy corollary we get that either the skew product is transitive or the nonwandering set has zero measure. Lyapunov stability of the attractor also applies to the interval fiber case, since any skew product with interval fiber can be continued to a skew product with circle fiber. However, it is unknown whether the attractor is asymptotically stable even for the interval fiber.

Main ingredients of the proof are
\begin{itemize}
\item using the semicontinuity lemma to get Lyapunov stability (as in~\cite{MP})
\item
the fact that the statistical $\omega$-limit set of a generic point is saturated by unstable leaves (\cite{BDV}).
\end{itemize}

The results of this paper also hold for step skew products, see~\cite{OSh} (in preparation).

\section{Definitions and results} \label{intro}
\begin{defin}[\cite{Mil}, p.180] \label{d:AM}
For a diffeomorphism $F$ of a riemannian manifold $X$, the \emph{Milnor attractor} (it is also called the \emph{likely limit set}) of $F$ (notation: $A_M(F)$) is the smallest (with respect to inclusion) closed subset of $X$ containing $\omega$-limit sets of Lebesgue-almost all points.
\end{defin}

\begin{defin}
We will call the \emph{statisitical $\omega$-limit set} of point $x$ (notation: $\oms(x)$) the union of all points $y \in X$, such that for any neighborhood $U$ of $y$
\[
\limsup_{N \to +\infty} \frac 1 N |\{n: F^n(x) \in U, 0 \le n < N \}| > 0.
\]
\end{defin}

\begin{defin}[\cite{AAISh}, $\S8.2$; see also~\cite{GI}] \label{d:As}
The \emph{statistical attractor} (notation: $\As(F)$) is defined exactly like the Milnor attractor, but $\omega$-limit set in the definition is replaced by statistical $\omega$-limit set.
\end{defin}

\begin{rem*}
The existence of Milnor attractor is proved in~\cite[Lemma~$1$]{Mil}. The existence of statistical attractor can be proved in the same way. The definition of statistical attractor in~\cite{AAISh} is a little different from stated above, but it is easy to see that these definitions are equivalent.
\end{rem*}

\begin{defin}
A subset $Y \sub X$ is \emph{Lyapunov stable} for a map $F: X \to X$ if for every neighborhood $U$
of $Y$, there is a neighborhood $V \subset U$ of $Y$ such that $F^n(V) \sub U$ for any $n \ge 0$.
\end{defin}

Let $B$ be a compact riemannian manifold. Fix a transitive Anosov diffeomorphism $A: B \to B$. Consider a product $X = B \times \ss$ (we will call $B$ \emph{the base} and $\ss$ \emph{the fiber}) and a skew product diffeomorphism
\[
F: X \to X \quad (x, y) \mapsto (A(x), f_x(y)).
\]
The map $f_x$ is called \emph{the fiber map} above $x$.

For $r \ge 2$ let $\P$ be the set of all such skew products that are $C^r$-smooth, partially hyperbolic in the narrow sense (see the definition in~\cite[\S2.2]{Pes}, narrow sense means that the invariant splitting is of type $E^u \oplus E^c \oplus E^s$) with central direction tangent to the skew product fibers, and all fiber maps $f_x(y)$ preserve the orientation of $\ss$. Endow $\P$ with metric
\[
\dist(F, G) = \dist_{C^r(X)}(F, G) + \dist_{C^r(X)}(F\m, G\m).
\]

\begin{Theorem A}
There is a residual subset $R \subset \P$ such that for all $F \in R$ the statistical attractor of $F$
\begin{itemize}
\item coincides with the Milnor attractor
\item is Lyapunov stable
\item either has zero Lebesgue measure or coincides with the whole phase space $X$.
\end{itemize}
\end{Theorem A}
\begin{corB}
For any $F \in R$ either the non-wandering set of $F$ has zero Lebesgue measure or $F$ is transitive and a generic point with respect to the Lebesgue measure has a dense orbit.
\end{corB}

\section{Sketch of the proof} \label{sketch}
The proof is based on two well-known ideas. The first one is Theorem~$11.16$ from~\cite{BDV}, which implies the following as an easy consequence:
\begin{lem}[consequence of \cite{BDV}, theorem~$11.16$] \label{sat}
For any $C^2$-smooth partially hyperbolic diffeomorphism with invariant splitting of type $E^u \oplus E^{cs}$ the statistical $\om$-limit set of almost any point with respect to the Lebesgue measure is saturated by the unstable leaves, i.e. if $x \in \oms(y)$, then $\wu(x) \subset \oms(y)$.
\end{lem}
Note that this result is applicable for systems with invariant splitting of type $E^u \oplus E^c \oplus E^s$, since we may take $E^{cs}=E^c \oplus E^s$. An analog of Lemma~\ref{sat} also holds for $\om$-limit sets, see~\cite{MO}.

The second idea is to use semicontinuity technique to obtain Lyapunov stability. We will prove an analog of the following theorem in $C^r$ for skew products with the fiber a circle.
\begin{thm}[\cite{MP}, Theorem~6.1] \label{MP}
For a $C^1$-topologically generic diffeomorphism and any periodic point $p$ the point $p$ is hyperbolic and the closure of the unstable manifold of $p$ is Lyapunov stable.
\end{thm}
\begin{rem*}
This theorem is stated in~\cite{MP} for singularities of vector fields. The formulation above is from~\cite[$\S3.1$]{ABD}.
\end{rem*}

Now we can give the plan of our proof. Assume for simplicity that the Anosov diffeomorphism $A$ in the base has a fixed point $p$. If this is not the case, we need to consider a periodic point and adapt some parts of the proof --- see section~\ref{nofixed}. Consider $f_p: \{p\} \times \ss \to \{p\} \times \ss$, the fiber map  above $p$ of the skew product $F$. For generic skew products the map $f_p$ is a Morse-Smale diffeomorphism and has finitely many attractors $a_i \in \{p\} \times \ss$.

Consider the closures of the unstable leaves of the points $a_i$ (notation: $\cu(a_i)$). Note that since $a_i$ is an attractor of $f_p$, the leaf $\wu(a_i)$ is also the unstable manifold of the periodic point $a_i$. Under some genericity assumption on the skew product we prove that for a Lebesgue-generic point $x$ the set $\oms(x)$ contains at least one of the points $a_i$. By Lemma~\ref{sat} the set $\oms(x)$ contains $\cu(a_i)$ as well. But by the analog of Theorem~\ref{MP} the set $\cu(a_i)$ is Lyapunov stable. Using this fact and the definition of Lyapunov stability, it is easy to prove that $\oms(x)=\cu(a_i)$. Thus $A_{stat}$ is a union of some of the sets $\cu(a_i)$ and therefore is Lyapunov stable. A Lyapunov stable $\As$ always coincides with $A_M$.

If $\As$ has positive measure, it must (since it is invariant) intersect the $\As$ of the inverse skew product, which is Lyapunov stable under the inverse system. Using the intersection, we show that for some $i$ the set $\cu(a_i) \subset \As$ coincides with $\cs(r_j)$ for some repelling periodic point $r_j$ of the map $f_p$. This set is Lyapunov stable for both $F$ and $F^{-1}$. Then it is easy to show that this set is the whole phase space.

\section{Notation}
We continue the notation introduced before the statement of Theorem~A.
\begin{description}
\item[$r=2, 3, \dots, \infty$] is a fixed number that denotes the smoothness of the class $C^r$ for which we prove the theorem.

\item[$\leb$] is the probability (i.e. $\leb(X)=1$) Lebesgue measure on $X$.

\item[$\ws (\wu, W^c, W^{cs}, W^{cu})$] denotes a stable (unstable, center, center-stable, center-unstable) leaf.

\item[$x \arr y,$] is written for $x, y \in X$ if arbitrary small neighborhoods of $x$ and $y$ are connected by a trajectory going from the neighborhood of $x$ to the neighborhood of $y$.

\item[$Sat(F) \subset X$] is the set of all points $x \in X$ such that $\oms(x)$ is saturated by the unstable leaves. By Lemma~\ref{sat} above $\leb(Sat)=1$.

\item[$p$:] we choose and fix any fixed point $p \in B$ of the Anosov diffeomorphism $A$.
 If $A$ has no fixed points, we need to consider a periodic point and adapt some parts of the proof --- see section~\ref{nofixed}.

\item[$\ss_p$] is the fiber $\{p\} \times \ss$ above the fixed point $p$.

\item[$f_p: \ss_p \to \ss_p$] is the fiber map above $p$ of the skew product $F$.

\item[$\ms \sub \P$] is the set of $C^r$-skew products, such that the map $f_p$ is Morse-Smale for the point $p$ fixed above. Since Morse-Smale maps form an open and dense subset of $\Diff_+^r(\ss)$, the set $\ms$ is an open and dense subset of $\P$.

\item[$a_i \in X$] are the attracting periodic points of the Morse-Smale map $f_p$.

\item[$r_i \in X$] are the repelling periodic points of $f_p$.

\item[$\pi_B: X \to B$] is the projection on the base along the fibers.
\end{description}

\section{Preliminaries}
Since $A$ is transitive, the unstable leaves of $A$ are dense in $B$ (see~\cite[\S2.1]{Pes}). Note that the (un)stable leaves of $F$ are mapped in the (un)stable leaves of $A$ by the projection $\pi_B$.

The following lemma easily follows from the definition of Lyapunov stability.
\begin{lem} \label{arr}
If $x \arr y$ and $x$ belongs to some closed Lyapunov stable set, then $y$ also belongs to this set.
\end{lem}

\begin{lem} \label{noescape}
Let $A \subset X$ be closed and Lyapunov stable for $F$. Then for any $x \in X$
\begin{itemize}
\item if $\om(x)$ intersects $A$, then $\om(x) \subset A$
\item if $\oms(x)$ intersects $A$, then $\oms(x) \subset A$.
\end{itemize}
\end{lem}
\begin{proof}
It is an immediate consequence of Lemma~\ref{arr}, since for any $x, y \in \om(x)$ (or $\oms(x)$) $x \arr y$.
\end{proof}

\begin{lem}
\label{amas}
If $\As$ is Lyapunov stable, then it coincides with $\A$.
\end{lem}
\begin{proof}
Note that $\As$ always is a subset of $\A$ (it was first proved in~\cite[ $\S8.2$]{AAISh}).  Indeed, for any point $x$ it is easy to see that $\oms(x) \subset \om(x)$, thus $\As \subset \A$.  Let us prove that $\A \subset \As$ if $\As$ is Lyapunov stable. By the definition of $\As$ for almost any point $x$ we have $\oms(x) \subset \As$, thus $\om(x)$ intersects $\As$. By Lemma~\ref{noescape} this implies $\om(x) \subset \As$.
\end{proof}

This means that we only need to prove that $\As$ is Lyapunov stable and either has zero measure or coincides with the whole phase space. So we study $\As$, and the word ``attractor'' will refer to $\As$.

We will often use that a $C^r$-small perturbation of a skew product is also small in the topology of $\P$. It follows from the following known consequence of the implicit function theorem (see~\cite[Lemma~$3.2$]{BRWZ}):
\begin{lem}\label{l:inverse}
The map $F \mapsto F^{-1}$ is a homeomorphism of $\Diff^r(M)$.
\end{lem}

\section{Attracting to $a_i$} 
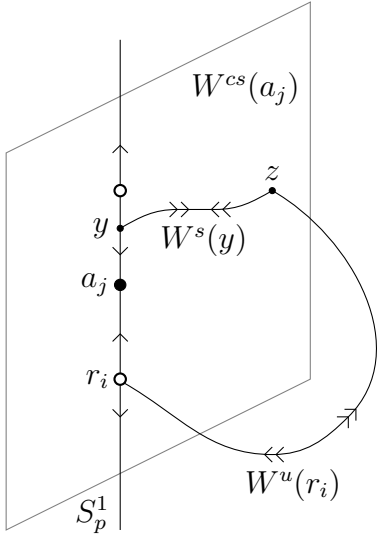
\label{s:toa}
\begin{wrapfigure}{l}[-2mm]{50mm}
\vspace{-10mm}
\begin{center}
\begin{tikzpicture}[scale=0.5]
\draw [thin, gray] (0,0) -- (8,4) -- (8,14) -- (0,10) -- cycle;
\draw [fill] (3, 6.5) circle[radius=0.15];
\draw [thick] (3, 4) circle[radius=0.15];
\draw [thick] (3, 9) circle[radius=0.15];
\draw [fill] (3, 8) circle[radius=0.08];
\draw (3, 0) -- (3, 3.85);
\draw (3, 4.15) -- (3, 8.85);
\draw (3, 9.15) -- (3, 13);

\draw (3.2, 3.2) to (3, 3) to (2.8, 3.2);
\draw (3.2, 7.5) to (3, 7.3) to (2.8, 7.5);
\draw (3.2, 5) to (3, 5.2) to (2.8, 5);
\draw (3.2, 10) to (3, 10.2) to (2.8, 10);

\node [left] at (3, 4) {$r_i$};
\node [left] at (3, 6.5) {$a_j$};
\node [left] at (3, 8) {$y$};
\node [left] at (3, 0.4) {$S^1_p$};
\node at (6.3, 11.7) {$W^{cs}(a_j)$};

\node [above] at (7, 9) {$z$};
\draw [fill] (7, 9) circle[radius=0.08];

\draw (3, 8) to [out=30, in=180] (4.5, 8.5) to (5.5, 8.5) to [out=0, in=210] (7, 9);
\draw (4.3, 8.7) to (4.5, 8.5) to (4.3, 8.3);
\draw (4.6, 8.7) to (4.8, 8.5) to (4.6, 8.3);
\draw (5.6, 8.7) to (5.4, 8.5) to (5.6, 8.3);
\draw (5.9, 8.7) to (5.7, 8.5) to (5.9, 8.3);
\node[below] at (5.2, 8.4) {$W^s(y)$};

\draw (7, 9) to [out=-30, in=45] (9, 3) to [out=-135, in=0] (7, 2) to [out=180, in=-30] (3.1, 3.9);
\node [right] at (6, 1.2) {$W^u(r_i)$};
\draw (7.0, 1.8) to (6.8, 2) to (7.0, 2.2);
\draw (7.3, 1.8) to (7.1, 2) to (7.3, 2.2);
\draw (8.7, 3) to (9, 3) to (9, 2.7);
\draw (8.9, 3.2) to (9.2, 3.2) to (9.2, 2.9);
\end{tikzpicture}
\end{center}
\vspace{-30mm}
\end{wrapfigure}

In this section we prove that under certain genericity assumption for any point $x \in Sat$ there is a number $j$ such that $a_j \in \oms(x)$. First we state this assumption (cf. the picture).
\begin{assu} \label{norep}
The map $f_p$ is Morse-Smale and for any repeller $r_i$ of $f_p$ there are
\begin{itemize}
\item a point $z$ on the unstable leaf of $r_i$
\item and a point $y$ on the stable leaf of $z$
\item such that $y \in \ss_p$ and $y$ is not a repeller of $f_p$.
\end{itemize}
\end{assu}

Then we prove that this assumption is generic.
\begin{lem}\label{gen}
Systems satisfying assumption~\ref{norep} form an open and dense subset of $\P$.
\end{lem}
\bp
We start with the main idea of the proof, ignoring some technical details.
Consider the stable leaf $\ws(p, A)$ of the point $p$ for the map $A$. Take an open ball $\ws_1(p, A)$ in this leaf with center $p$ and radius $1$. Define the set
$$T=\ws_1(p, A) \times \ss, T \subset W^{cs}(a_j).$$
Note that $T$ is transversal to the unstable leaves of $F$ and is foliated by the local stable leaves of the points of $\ss_p$. Fix any intersection point
$$z_B \in \ws_1(p, A) \cap (\wu(p, A) \setminus \{p\}).$$
Such point exists since the unstable leaves of $A$ are dense. Given $F$ and $r_i$, we define a map $\pi_B^{-1}: W^u(p, A) \mapsto W^u(r_i)$ that maps a point of $W^u(p, A)$ to the unique point of $W^u(r_i)$ above it. Set $z=\pi_B^{-1}(z_B)$, then $z \in W^u(r_i) \cap T$.

Now we can consider the following condition:
\begin{equation} \tag{*}
z \text{ is not in a local stable leaf of some repeller of } f_p.
\end{equation}
Note that $*$ implies assumption~\ref{norep} with $y=W^s_{loc}(z) \cap \ss_p$. The fact that $*$~is an open and dense condition is obvious, because the union of the local stable leaves of the repellers of $f_p$ has codimension at least one as a subset of $T$.

Now we give a formal proof of the lemma using the notations above. It is enough to prove that for any $F_0 \in \ms$ assumption~\ref{norep} holds on an open and dense subset of some small neighborhood $U \subset \P$ of $F_0$. Take the neighborhood $U$ so small that $a_i(F_0)$ can be continued on $U$ and no new attractors of $f_p$ appear. Now we fix $i$ and check that assumption~\ref{norep} is satisfied for this $i$ on an open and dense subset of $U$.

We assume that $A^{-1}(z_B) \notin \ws_1(p, A)$, replacing $z_B$ by $A^{-l}(z_B)$ for some $l$ if necessary.
Now for any $F \in U$ we may consider the point $z$ as above, and state condition~$*$.

Openness of condition~$*$ is obvious. To prove its density, we perturb the fiber map $f_c$, where $c=A^{-1}(z_B)$, in such way that $f_d$ is fixed
\begin{itemize}
\item
for any point $d \in \pi_B(T)$,
\item
for $d$ close to $p$,
\item
for $d=A^{-k}(z_B), k \ge 2$.
\end{itemize}

Since $c \notin \ws_1(p, A)$, such perturbation preserves $W^s(r_j) \cap T$ for all $j$ by the first condition. The second condition means that a small piece of  $W^u(r_i)$ around $r_i$ is preserved. Since the whole $W^u(r_i)$ is obtained by iterating forward this piece, the third condition implies that the point $F^{-1}(z) \in W^u(r_i)$ is also preserved. The fiber coordinate of $z$ is $f_c(t)$, where $t$ is the fiber coordinate of $F^{-1}(z)$. Thus we can move the point $z$ away from the stable manifolds of the repellers by perturbing the fiber map $f_c$.
\ep

And now we can prove the required statement.
\begin{lem} \label{toa}
If assumption~\ref{norep} holds, then for any point $x \in Sat$ there is a number $j$ such that $a_j \in \oms(x)$.
\end{lem}
\bp
The projections on the base of the unstable leaves of $F$ are the unstable leaves of $A$ and thus are dense in $B$. Hence if $x \in Sat$, then $\oms(x)$ intersects $\ss_p$. There are two possibilities.

Case 1. The intersection $\oms(x) \cap \ss_p$ containts a point $q$ that is not a repeller of $f_p$. Then for some $j$ the sequence $s_n=F^{n}(q), s_n \in \ss_p$ has $a_j$ as a limit point. Since $\oms(x)$ is invariant and closed, we see that $a_j \in \oms(x)$.

Case 2. The intersection $\oms(x) \cap \ss_p$ containts only the repellers of $f_p$. Let $r_i \in \oms(x)$. Using assumption~\ref{norep}, we get the points $z$ and $y$. Since $x \in Sat$ and $z \in W^u(r_i)$, we have $z \in \oms(x)$. The sequence $s_n=F^{n}(y)$ has some $a_j$ as a limit point. Since $\dist(F^n(y), F^n(z)) \to 0$,  $a_j$ is also a limit point of the sequence $s'_n=F^{n}(z)$. As above, it follows that $a_j \in \oms(x)$.
\ep

\section{$\cu(a_i)$ are Lyapunov stable} \label{s:Wu}
The goal of this section is to prove that the assumption in the name of the section is generic.
First we give the precise statements.
\begin{assu} \label{wust}
The map $f_p$ is Morse-Smale and the set $\cu(a_i)$ is Lyapunov stable for any attractor $a_i$ of $f_p$.
\end{assu}

\begin{lem} \label{Wu}
Skew products satisfying assumption~\ref{wust} form a residual subset of $\P$.
\end{lem}
\begin{proof}
Note that $\wu(a_i)$ coincides with the unstable manifold of the periodic point $a_i$, so Lemma~\ref{Wu} is very similar to Theorem~\ref{MP}. The difference is that $C^1$ is replaced by $C^r$, the statement is a little weaker and is claimed only for skew products with circle fiber. This lemma is proved in four steps. The first three steps exactly mimic the proof of Theorem~\ref{MP}, so we will just sketch them.

Step 1. Lemma~\ref{Wu} is reduced to a local version, stating that for any $F \in \ms$ there is a small neighborhood $U \ni F$, such that diffeomorphisms satisfying assumption~\ref{wust} form a residual subset of $U$. Since $f_p$ is Morse-Smale, the points $a_i$ survive in $U$ for $U$ small enough, and no new $a_i$ appear. So for any $i$ we can consider a set-valued function $\cu(a_i):~U~\to~\mathcal{K}(X)$, where $\mathcal{K}(X)$ is the set of all compact subsets of $X$, endowed with Hausdorff metric, and $a_i$ is the continuation of the periodic saddle $a_i(F)$.

Step 2. Note that these functions are lower-semicontinuous (see~\cite[$\S2.5$]{ABC} for the definition of semicontinuity). This follows from the fact that any compact part of $\wu(a_i)$ continuously depends on the map (recall that $\wu(a_i)$ coincides with the unstable manifold of periodic point $a_i$).

Step 3. A standard result in general topology (the semicontinuity lemma, see~\cite[$\S2.5$]{ABC}) states that continuity points of a lower-semicontinuous function (with values in the set of closed subsets of a manifold) form a residual subset.  Thus, $\cu(a_i)$ depend on the map continuously on a residual subset of $U$. This holds in $C^r$ for any $r \ge 1$.

Step 4. We show that if $\cu(a_i)$ is Lyapunov unstable, it depends on the map discontinuously. This is done in the proof of Theorem~\ref{MP} for $r=1$ using the connecting lemma. However, we need $r>1$ in order to use Lemma~\ref{sat}. Lemma~\ref{pert} below makes the last step for skew products with one-dimensional fiber for any $r$, using a monotonicity argument instead of the connecting lemma. Thus, Lemma~\ref{Wu} has been reduced to Lemma~\ref{pert}.
\end{proof}

\begin{lem} \label{pert}
Let $F \in \ms$. Assume that for some attractor $a$ of $f_p$ the set $\cu(a)$ is Lyapunov unstable. Then this set depends discontinuously on the skew product at the point $F$.
\end{lem}

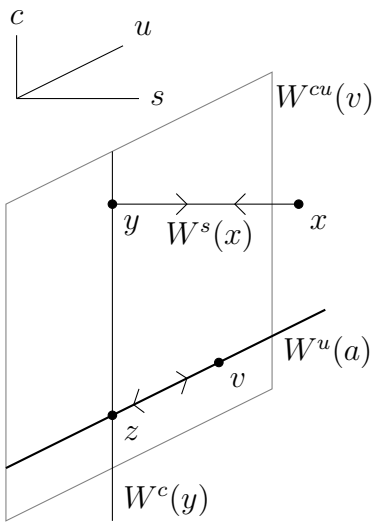
\begin{wrapfigure}{l}[-2mm]{50mm}
\vspace{-10mm}
\begin{center}
\begin{tikzpicture}[scale=0.7]
\draw [thin, gray] (0,4) -- (5,6.5) -- (5,12.5) -- (0,10) -- cycle;
\begin{scope}[shift={(0.2, 12)}, scale=1]
\draw (0, 0) -- (2, 1);
\node[above right] at (2, 1){$u$};
\draw (0, 0) -- (2.3, 0);
\node[right] at (2.3, 0){$s$};
\draw (0, 0) -- (0, 1.2);
\node[above] at (0, 1.2){$c$};
\end{scope}
\draw (2, 4) -- (2, 11);
\node [right] at (2, 4.4) {$W^c(y)$};
\node [right] at (4.85, 12){$W^{cu}(v)$};
\begin{scope}[shift={(0, -3)}, scale=1]
\draw [thick] (0, 8) -- (6, 8+6/2);
\node [below right] at (5, 10.7){$W^u(a)$};
\draw [fill] (2, 9) circle[radius=0.08];
\node [below right] at (2, 9) {$z$};
\draw (2.6, 9.05) -- (2.4, 9.2) -- (2.5, 9.5);
\draw (3.2, 9.83) -- (3.4, 9.7) -- (3.3, 9.4);
\draw [fill] (4, 10) circle[radius=0.08];
\node [below right] at (4, 10) {$v$};
\end{scope}
\begin{scope}[shift={(0, 4)}, scale=1]
\draw [fill] (2, 6) circle[radius=0.08];
\node [below right] at (2, 6) {$y$};
\node [below right] at (2.8, 5.9) {$W^s(x)$};
\draw (2, 6) -- (5.5, 6);
\draw (3.2, 6.2) -- (3.4, 6) -- (3.2, 5.8);
\draw (4.5, 6.2) -- (4.3, 6) -- (4.5, 5.8);
\draw [fill] (5.5, 6) circle[radius=0.08];
\node [below right] at (5.5, 6) {$x$};
\end{scope}
\end{tikzpicture}
\end{center}
\vspace{-15mm}
\end{wrapfigure}

\noindent\textit{Proof.}
Since $\cu(a)$ is Lyapunov unstable, there is a real number $c>0$ such that there is a point $x \in X$ arbitrary close to $\cu(a)$ that runs $c$-away from this set:
\begin{itemize}
\item there are two points $x$ and $v$ arbitrary close, $v \in W^u(a)$,
\item $\dist(F^n(x), \cu(a)) > c$ for some $n \in \mathbb N$.
\end{itemize}

Let us take a new run-away point $y=W^{cu}_{loc}(v) \cap W^s_{loc}(x)$ (cf. the picture). It is on the same fiber as the point $z=W^u_{loc}(v) \cap W^c_{loc}(y), z \in W^u(a)$.

We may assume that $x$ and $v$ are very close. Then the local leaves in consideration are small and almost straight, and $n$ is large. Then $y$ is close to $z$. Since $x$ and $y$ are connected by a small arc of a stable leaf and $n$ is large, $\dist(F^n(y), F^n(x))~<~c/2.$

Thus we obtained two points $y$ and $z$ arbitrary close on the same fiber with
$$z \in W^u(a), \dist(F^n(y), \cu(a)) > c/2.$$

Now we use this pair of points to construct a skew product $\tilde F$ such that
\begin{itemize}
\item
$\tilde F$ is $\dist(y, z)$-close to $F$ in $C^r$ topology
\item
$F^n(y) \in W^u_{\tilde F}(\tilde a)$, where $\tilde a$ is the continuation of $a$.
\end{itemize}
Since $F^n(y)$ is at least $c/2$-away from $\cu(a)$, so is $\overline{W^u_{\tilde F}}(\tilde a)$. Thus existence of such $\tilde F$ implies that $\cu(a)$ depends discontinuously on the skew product.

Given $f:\ss \mapsto \ss$, define $f+b=R_b \circ f$, where $R_b$ is the rigid rotation by angle $2 \pi b$. Consider a family of skew products $F_b$ obtained by adding $b$ to all fiberwise maps of $F$:
\[
F_b: X \to X \quad (x, y) \mapsto (A(x), (f_x+b)(y)) \quad b \ge 0.
\]

Let us prove by the intermediate value theorem that for some $b \in [-\eps, \eps]$ we have $F^n(y) \in W^u_{F_b}(a_b)$, where $a_b$ is the continuation of $a$. This will give the map $\tilde F$ we seek.

We lift the restriction of $F_b$ on $\wu_A(p) \times \ss$ to a continuous map $$\hat F_b: \wu_A(p) \times \RR \to \wu_A(p) \times \RR$$
 in such way that $\hat F_b$ continuously depends on $b$. Set $\hat F=\hat F_0$. Let us also lift the points $a_b, z, y$ to the points $\hat a_b, \hat z, \hat y$ in such way that $\hat z \in \wu_{\hat F}(\hat a_0), \dist(\hat z, \hat y)=\dist(z, y)$.

Denote by $\hat z_b$ the unique point of $W^u_{\hat F_b}(\hat a_b)$ on the fiber of $\hat z$. What we need to prove is that for some $b$ the points $\hat z_b$ and $\hat y_b=\hat F_b^{-n}(\hat F^n(y))$ coincide. They lie on the same fiber $\mathbb R$ for any $b$. For $b=0$ they are $\eps$-close, assume WLOG that $\hat z_0<\hat y_0<\hat z_0+\eps$.

Now let us replace $b=0$ by $b=\eps$. The set $\wu_{\hat F_b}(\hat a_b)$ is as a plot of a function from $W^u_A(p) \subset B$ to the fiber $\mathbb R$. After the change of $b$ the point $\hat a_b$ will stay on its center fiber and move in the positive direction. Thus the value of this function will increase at the points corresponding to a small piece of $\wu_{\hat F_b}(\hat a_b)$ around $a_b$ . Since the whole $\wu_{\hat F_b}(\hat a_b)$ is obtained by iterating forward this small piece, and the fiberwise maps are increased by $\eps$, the value of this function will increase by at least $\eps$ at any point. Thus $\hat z_\eps > \hat z_0+\eps$. The point $\hat y_b$ will move in the negative direction ($\hat y_\eps < \hat y_0$), since when we increase the fiberwise maps, their inverse maps decrease. So for $b=\eps$ we have $\hat y_\eps < \hat z_\eps$. To finish the proof, we use the intermediate value theorem.
\qed

\section{Proof of Theorem A}
Let us start by defining the residual subset $R \subset \P$ on which the Theorem~A holds. The set $R$ is formed by all skew products $F \in \P$ such that
\begin{itemize}
\item \label{cond1}
$F \in \ms$, and for any attractor $a_i$ of the map $f_p$ the set $\cu(a_i)$ is Lyapunov stable,
\item \label{cond2}
for any $x \in Sat$ there is a number $i$, such that $a_i \in \oms(x)$,
\item
these two properties also hold for the inverse skew product $F^{-1}$.
\end{itemize}

The first two properties give a residual subset by sections~\ref{s:toa} and~\ref{s:Wu}.
Since $F \mapsto F^{-1}$ is a homeomorphism between $\P(A)$ and $\P(A^{-1})$ by Lemma~\ref{l:inverse}, the third condition also defines a residual subset.

Let us prove that the attractor is Lyapunov stable.

\begin{lem}
\label{ast}
For any $F \in R$
\begin{enumerate}
\item \label{i2}
for any $x \in Sat$ there is a number $i$, such that $\oms(x)=\cu(a_i)$,
\item \label{i3}
denote by $I$ the set of numbers $i$ such that $\leb(\{x: \oms(x)=\cu(a_i)\})>0$. Then
$$\As=\cup_{i \in I} \cu(a_i),$$
\item \label{i35}
for any $i \in I$ the set $\cu(a_i)$ is weakly transitive, i.e. for any $y, z \in \cu(a_i)$ we have $y \arr z$,
\item \label{i4}
$\As$ is Lyapunov stable.
\end{enumerate}
\end{lem}
\emph{Proof.}
\begin{enumerate}
\item Consider any $x \in Sat$. Since $F \in R$ the set $\oms(x)$ contains at least one of the $a_i$. Since $\oms(x)$ is saturated by the unstable leaves, we have $\cu(a_i) \subset \oms(x)$. The set $\cu(a_i)$ is Lyapunov stable, so $\oms(x) = \cu(a_i)$ by Lemma~\ref{noescape}.

\item By Lemma~\ref{sat} $\leb(Sat)=1$. So this follows from item~\ref{i2}.

\item Consider a point $x$ with $\oms(x)=\cu(a_i)$. Since the orbit of $x$ comes arbitrary close to both $y$ and $z$, $y \arr z$.

\item This follows from item~\ref{i3} since the sets $\cu(a_i)$ are Lyapunov stable.
\qed
\end{enumerate}


\begin{rem*}
The sets $\cu(a_i)$ may coincide, include each other, or intersect. However, the sets $\cu(a_i)$ that belong to the attractor either coincide or are disjoint. We do not prove this remark, since it is not used below.
\end{rem*}

Let us prove that the attractor is not thick.

\begin{lem} \label{thin}
For any $F \in R$ either $\leb(\As(F))=0$ or $\oms(x)=X$ for $\leb$-almost any $x \in X$.
\end{lem}

\begin{proof}
Assume that $\leb(\As(F))>0$. Since the attractor is closed and invariant, it contains statistical $\alpha$-limit sets (i.e. statistical $\omega$-limit sets under $F^{-1}$) of all its points. Since these points form a set of positive measure, $\As(F)$ intersects with $\As(F^{-1})$. Applying item~\ref{i3} of Lemma~\ref{ast} to both $F$ and $F^{-1}$, we see that for some attractor $a$ and repeller $r$ (this $r$ overrides the $r$ introduced in the notation section for the duration of this proof) of the map $f_p$ the sets $\cu(a)$ and $\cs(r)$ intersect. Denote any intersection point of this sets by $c$. Recall that by item~\ref{i3} of Lemma~\ref{ast} the set $\cu(a)$ is Lyapunov stable, while $\cs(r)$ is Lyapunov stable for $F^{-1}$.

Let us prove that $\cs(r) \subset \cu(a)$. Consider any $x \in \cs(r)$. Applying item~\ref{i35} of Lemma~\ref{ast} to $F^{-1}$, we see that $c \arr x$. Since $c \in \cu(a)$, Lemma~\ref{arr} implies that $x \in \cu(a)$. Replacing $F$ by $F^{-1}$ in the previous argument, we get $\cu(a) \subset \cs(r)$, so these two sets coincide. Denote $Y=\cu(a)=\cs(r)$.

The set $Y$ is forward and backwards Lyapunov stable, non-empty, closed, and invariant. So is the set $Y_p := Y \cap \ss_p$ (under the action of the Morse-Smale map $f_p$). Lemma~\ref{fbst} below states that any such set is equal to the whole circle: $Y_p=\ss_p$. Lemma~\ref{stsat} below states that any Lyapunov stable invariant set is saturated by unstable leaves. Hence, $Y$ is saturated by both stable and unstable leaves. We claim that any point $x \in X$ can be connected with a point of $\ss_p$ by a path formed by stable and unstable leaves. This implies that $Y=X$. To prove the claim, connect the points $\pi_B x$ and $p$ by a path formed by stable and unstable leaves of $A$ in the base.  Since (un)stable leaves of $F$ project to (un)stable leaves of $A$ by $\pi_B$, this path can be lifted to a path connecting $x$ with some point of $\ss_p$, formed by stable and unstable leaves of $F$.

Since $X=Y=\cu(a)$, item~\ref{i35} of Lemma~\ref{ast} means that for any $y, z \in X$ we have $y \arr z$. Then Lemma~\ref{arr} implies that for any $i$ we have $\cu(a_i)=X$. By item~\ref{i2} of Lemma~\ref{ast} for any $p \in Sat$ we have $\oms(p)=X$. Since $\leb(Sat)=1$, a $\leb$-generic point has a dense orbit.
\end{proof}

\begin{lem} \label{fbst}
Let $g$ be a Morse-Smale circle diffeomorphism of $\ss$, let $Y \subset \ss$ be forward and backward Lyapunov stable, non-empty, closed, and invariant. Then $Y=\ss$.
\end{lem}
\begin{proof}
Since $\ss$ is the only non-empty open and closed subset of $\ss$, it is enough to prove that $Y$ is open. To do so, take any $y \in Y$. If $y$ is not a repeller of $g$, it belongs to a basin of attraction of some attractor $a$. Since $y \arr a$, by Lemma~\ref{arr} we have $a \in Y$. Take any point $b$ in the basin of attraction of $a$. Since $b \arr a$, by Lemma~\ref{arr} we have $b \in Y$, thus $Y$ contains the whole basin of attraction of $a$, and $y$ is an interior point of $Y$. If $y$ is a repeller, similar argument shows that the whole basin of repulsion of $y$ belongs to $Y$, thus $y$ is an interior point again.
\end{proof}

\begin{lem} \label{stsat}
Let $A \subset X$ be a Lyapunov stable closed invariant set. Then $A$ is saturated by unstable leaves, i.e. $x \in A$ implies $\wu(x) \in A$.
\end{lem}
\begin{proof}
Let $x \in A$, $y \in \wu(x)$. Take any limit point $z$ of the sequence $F^{-n}(x)$. Since $y \in \wu(x)$, $\dist(F^{-n}(x), F^{-n}(y)) \to 0$, so $z \arr y$. Since $A$ is closed and invariant, $z \in A$. Lemma~\ref{arr} implies that $y \in A$.
\end{proof}

\begin{proof}[Proof of Theorem A]
Take the set $R \subset \P$ defined at the beginning of this section. For any $F \in R$ the set $\As(F)$ is Lyapunov stable by Lemma~\ref{ast} and coincides with $\A(F)$ by Lemma~\ref{amas}. By Lemma~\ref{thin} we see that either $\leb(\As(F))=0$ or $\As(F)=X$.
\end{proof}

\section{Proof of Corollary B}
\begin{lem}\footnote{This lemma is a slightly improved version of the following statement by S. Minkov: if the Milnor attractor has positive measure, the attractor of the inverse diffeomorphism either has positive measure or is Lyapunov unstable.}
\label{nw}
Let $F$ be any diffeomorphism, such that $\A(F)$ is Lyapunov stable and $\leb(\A)=0$. Then $\leb(\nw(F))=0$, where $\nw$ denotes the non-wandering set.
\end{lem}
\begin{proof}
Let $B$ be the basin of attraction of $\A$, i.e. $B=\{ x\in X: \om(x) \sub \A \}$.
We claim that if $\A$ is Lyapunov stable, any point $x \in B \setminus \A$ is wandering. Assume the contrary. Take any $y \in \om(x) \sub \A$. Using that $x$ is non-wandering, it is easy to see that $y \arr x$. By Lemma~\ref{arr} we have $x \in \A$, which is a contradiction.

Now note that $\leb(B)=1$ (by the definition of $\A$), so if $\leb(\A)=0$, then $\leb(B \setminus \A)=1$. Since $\nw$ does not intersect $B \setminus \A$, we have $\leb(\nw)=0$.
\end{proof}

\begin{proof}[Proof of Corollary B]
By Lemma~\ref{thin} for any $F \in R$ either a Lebesgue-generic point has a dense orbit or $\leb(\As)=0$. In the latter case, since $\As$ is equal to $\A$ and Lyapunov stable (by Theorem A), Lemma~\ref{nw} implies that $\leb(\nw(F))=0$.
\end{proof}

\section{What if $A$ has no fixed points?} \label{nofixed}
We started our proof by taking a fixed point $p$ of the Anosov diffeomorphism $A$. If there is no such fixed point, we consider a periodic point instead. The proofs above work in this situation with some minor changes that are hinted below. However, the author is not aware of any examples of Anosov differomorphisms on connected manifolds without a fixed point.
\begin{itemize}
\item The point $p$ is now periodic, not fixed.
\item Instead of the fiber map $f_p$ we need to consider $g_p$, the fiber map of $F^{\per(p)}$ above $p$:
\[g_p=f_{A^{\per(p)-1}(p)} \circ \dots \circ f_p.
\]
\item Recall that the set $\ms$ is formed by skew products, such that the map $g_p$ is Morse-Smale. We need to prove that the set $\ms$ is an open and dense subset of $\P$.
This follows from the fact that Morse-Smale maps form an open and dense subset of $\Diff_+^r(\ss)$. Openness of $\ms$ is obvious. To prove density, note that we can perturb $g_p$ as we want by perturbing the fiber map $f_p$.
\item In the proof of Lemma~\ref{toa} we should consider the map $F^{\per(p)}$ instead of $F$.
\item In the proof of Lemma~\ref{pert} we need to consider the dynamics above the unstable manifold of the whole orbit of the point $p$ in the base, not just $W^u_A(p)$.
\end{itemize}

\section{Acknowledgements}
The author is grateful to professor Yu.S. Ilyashenko for constant attention to this work and to I. Shilin and S. Minkov for useful discussions.


\begin{thebibliography}{99}
\bibitem[ABC]{ABC}
F. Abdenur, C. Bonatti, S. Crovisier,
\emph{Nonuniform hyperbolicity for $C^1$-generic diffeomorphisms,}
Israel J. Math. 183 (2011), 1-60

\bibitem[ABD]{ABD}
F. Abdenur, C. Bonatti, L.J. Diaz,
\emph{Non-wandering sets with non-empty interiors,}
Nonlinearity, 2004

\bibitem[AAISh]{AAISh}
V.I. Arnol'd, V.S. Afraimovich, Y.S. Ilyashenko, L.P. Shilnikov,
\emph{Bifurcation theory,}
Itogi Nauki i Tekhniki. Seriya "Sovremennye Problemy Matematiki. Fundamental'nye Napravleniya", 1986

\bibitem[BRWZ]{BRWZ}
 M. S. Baouendi, L. P. Rothschild, J. Winkelmann, D. Zaitsev,
 \emph{Lie group structures on groups of diffeomorphisms and applications to CR manifolds,}
 Annales de l'Institut Fourier, 2004

\bibitem[BP]{BP}
L. Barreira, Ya.B. Pesin,
\emph{Nonuniform Hyperbolicity: Dynamics of Systems with Nonzero Lyapunov Exponents,}
Cambridge University Press, 2007

\bibitem[BDV]{BDV}
C. Bonatti, L.J. Diaz, M. Viana,
\emph{Dynamics Beyond Uniform Hyperbolicity,}
Springer, 2004

\bibitem[CM]{CM}
C. M. Carballo, C. A. Morales,
\emph{Homoclinic classes and finitude of attractors for
vector fields on $n$-manifolds,}
Bulletin of the London Mathematical Society, 2003

\bibitem[GI]{GI}
A. Gorodetski, Yu. Ilyashenko,
\emph{Minimal and strange attractors,}
Int. Journ. of Bif. and Chaos 6 (1996) 1177–1183

\bibitem[Ily]{Ily}
Yu.S. Ilyashenko,
\emph{Thick attractors of boundary preserving diffeomorphisms,}
Indagationes Mathematicae, Vol. 22, Issues 3–4 (2011), 257–314

\bibitem[IKS]{IKS}
Yu. Ilyashenko, V. Kleptsyn, P. Saltykov,
\emph{Openness of the set of boundary preserving maps of an annulus with intermingled attracting basins,}
Journal of Fixed Point Theory and Applications, 2008

\bibitem[ISh]{ISh}
Yu. S. Ilyashenko, I. S. Shilin,
\emph{Relatively unstable attractors,}
Proceedings of the Steklov Institute of Mathematics, 2012

\bibitem[Kan]{Kan} 
I. Kan, 
\emph{Open sets of diffeomorphisms having two attractors, each with everywhere dense
basin,} 
Bull. Amer. Math. Soc. 31 (1994), 68-74 

\bibitem[KV]{KV}
V. Kleptsyn, D. Volk,
\emph{Physical measures for nonlinear random walks on interval,}
Mosc. Math. J., Volume 14,	Number 2, 339–365 (2014)

\bibitem[KV2]{KV2}
V. Kleptsyn, D. Volk,
\emph{Nonwandering sets of interval skew products,}
Nonlinearity 27 1595 doi:10.1088/0951-7715/27/7/1595 (2014)

\bibitem[Kud]{Kud}
Yu.G. Kudryashov,
\emph{Bony attractors,}
Funct. Anal. Appl., 44:219–222, 2010.
 
\bibitem[Kud2]{Kud2}
Yu. G. Kudryashov, 
\emph{Des orbites périodiques et des attracteurs des systémes dynamiques,}
PhD thesis, ENS Lyon, December 2010.

\bibitem[Mil]{Mil}
J. Milnor,
\emph{On the concept of attractor,}
Comm. Math. Phys. Volume 99, Number 2 (1985), 177-195

\bibitem[MO]{MO}
S. Minkov, A. Okunev,
\emph{Omega-limit sets of generic points for $E^u \oplus E^{cs}$-partially hyperbolic diffeomorphisms,}
submitted to Functional Analysis and Its Applications

\bibitem[MP]{MP}
C.A. Morales, M.J. Pacifico,
\emph{Mixing attractors for $3$-flows,}
Nonlinearity 14, 359-378 (2001)

\bibitem[OSh]{OSh}
A. Okunev, I. Shilin,
\emph{Milnor attractors of step skew products with one-dimensional fiber,}
in preparation

\bibitem[Pes]{Pes}
Ya. Pesin,
\emph{Lectures on partial hyperbolicity and stable ergodicity},
European Mathematical Society, 2004

\bibitem[Shi]{Shi}
I. Shilin,
\emph{Locally topologically generic diffeomorphisms with Lyapunov unstable Milnor attractors,}
in preparation

\end{thebibliography}
 \end{document}